\documentclass[12pt]{article}
\usepackage{amsmath,amsthm,amssymb}

\usepackage{mathrsfs}

\begin{document}
\newtheorem{lem}{Lemma}[section]
\newtheorem{theorem}{Theorem}[section]
\newtheorem{prop}{Proposition}[section]
\newtheorem{cor}{Corollary}[section]

\theoremstyle{remark}
\newtheorem{remark}{Remark}[section]

\allowdisplaybreaks

\newcommand{\la}{\langle}
\newcommand{\ra}{\rangle}
\newcommand{\eps}{\varepsilon}

\makeatletter\@addtoreset{equation}{section}\makeatother
\def\theequation{\arabic{section}.\arabic{equation}}

{\Large \bf
\begin{center} A note on equilibrium Glauber and
Kawasaki dynamics for permanental point processes
\end{center}
}

{\large Guanhua Li}\\ Department of Mathematics,
University of Wales Swansea, Singleton Park, Swansea SA2 8PP, U.K.\\
e-mail: \texttt{206674@swansea.ac.uk}\vspace{2mm}

{\large Eugene Lytvynov}\\ Department of Mathematics,
University of Wales Swansea, Singleton Park, Swansea SA2 8PP, U.K.\\
e-mail: \texttt{e.lytvynov@swansea.ac.uk}\vspace{2mm}
{\small

\begin{center}
{\bf Abstract}
\end{center}

We construct two types of equilibrium dynamics of an infinite particle system in a
locally compact metric space $X$ for which a permanental point process is a symmetrizing, and hence invariant measure.
The Glauber dynamics is a birth-and-death process in $X$, while
in the Kawasaki dynamics interacting particles randomly hop over $X$.
In the case $X=\mathbb R^d$, we consider a diffusion approximation for the Kawasaki dynamics at the level of Dirichlet forms. This  leads us to an equilibrium dynamics of interacting Brownian particles for which a permanental point process is a symmetrizing  measure.
\noindent  } \vspace{2mm}

\vspace{1.5mm}

 \vspace{1.5mm}

\section{Introduction}
Let $X$ be a locally compact Polish space and let $\nu$ be a Radon
non-atomic measure on it. Let $\Gamma=\Gamma_{X}$ denote the space
of all locally finite subsets (configurations) in $X$.

A Glauber dynamics (a birth-and-death process of an infinite system
of particles in $X$) is a Markov process on $\Gamma$ whose formal
(pre-)generator has the form
\begin{equation}\label{1}
\begin{aligned}
(L_{\mathrm G}F)(\gamma)&=\sum_{x\in{\gamma}}d(x,\gamma\setminus x)(F(\gamma\setminus
x)-F(\gamma))
\\
&  +\int_{X}\nu(dx)\,b(x, \gamma)\left( F(\gamma\cup
x)-F(\gamma)\right), \quad \gamma\in\Gamma.
\end{aligned}
\end{equation}
Here and below, for simplicity of notation we write $x$ instead of
$\{x\}$. The coefficient $d(x,\gamma\setminus x)$ describes the
rate at which particle $x$ of configuration $\gamma$ dies, while $b(x,\gamma)$ describes the rate at which, given configuration $\gamma$, a new particle
is born at $x$.

A Kawasaki dynamics (a dynamics of hopping particles) is a Markov
process on $\Gamma$ whose formal (pre-)generator is
\begin{equation}\label{2}
(L_{\mathrm K}F)(\gamma)=\sum_{x\in \gamma}c(x,y,\gamma\setminus
x)(F(\gamma\setminus x\cup y)-F(\gamma)), \quad \gamma\in\Gamma.
\end{equation}
The coefficient $c(x, y, \gamma\setminus x)$ describes the rate at
which particle $x$ of configuration $\gamma$ hops to $y$, taking the
rest of the configuration, $\gamma\setminus x$, into account.

Equilibrium Glauber and Kawasaki dynamics which have a standard
Gibbs measure as symmetrizing (and hence invariant) measure were
constructed in \cite{KL,KLR}. In \cite{LO}, this construction was extended
to the case of an equilibrium dynamics which has a determinantal
(fermion) point process as invariant measure, For further stu\-dies of
equilibrium and non-equilibrium Glauber and Kawasaki dynamics, we
refer to \cite{BCC,FK,FKK,FKL,FKO,HS,G1,G2,KKL,KKM,KKP,KKZ,KMZ,P} and the references therein.

The aim of this note is to show that general criteria of existence
of Glauber and Kawasaki dynamics which were developed in \cite{LO} are appliable to a wide class of
\mbox{$\alpha$-per}ma\-nental
($\alpha\in \mathbb{N}$) point processes, proposed  by Shirai and
Takahashi \cite{ST}. This class includes classical permanental (boson)
point processes, see e.g.\  \cite{DVJ,ST}. We will also consider a diffusion approximation for
the Kawasaki dynamics at the level of Dirichlet forms (compare with \cite{KKL}).
This will lead us to an equilibrium dynamics of interacting Brownian particles for which
an $\alpha$-permanental point process is a symmetrizing  measure. As a by-product of our considerations, we will also extend the result  of
\cite{ST} on the existence of $\alpha$-permanental point process.

\section{Equilibrium Glauber and Kawasaki dynamics -- general results}
Let $X$ be a locally compact Polish space. We denote by
$\mathcal{B}(X)$ the Borel $\sigma$-algebra on $X$, and by
$\mathcal{B}_{0}(X)$ the collection of all sets from
$\mathcal{B}(X)$ which are relatively compact. We fix a Radon,
non-atomic measure on $(X,\mathcal{B}(X))$. (For most applications,
the reader may think of $X$ as $\mathbb{R}^{d}$ and $\nu$ as the
Lebegue measure.)

The configuration space $\Gamma$ over $X$ is defined as the set of
all subsets of $X$ which are locally finite
 $$\Gamma:=\big\{\gamma\subset X : \,
|\gamma_\Lambda|<\infty\text{ for each  }\Lambda\in{\mathcal B}_0(X
)\big\},$$ where $|\cdot|$ denotes the cardinality of a set and
$\gamma_\Lambda:= \gamma\cap\Lambda$. One can identify any
$\gamma\in\Gamma$ with the positive Radon measure
$\sum_{x\in\gamma}\varepsilon_x$,  where
$\varepsilon_x$ is the Dirac measure with mass at $x$ and
$\sum_{x\in\varnothing}\varepsilon_x{:=}$zero measure. The space $\Gamma$ can be endowed with the vague topology,
i.e., the weakest topology on $\Gamma$ with respect to which  all
maps $$\Gamma\ni\gamma\mapsto\la \varphi,\gamma\ra:=\int_{X}
\varphi(x)\,\gamma(dx) =\sum_{x\in\gamma}\varphi(x),\quad
\varphi\in C_0(X),$$ are continuous. Here, $C_0(X)$ is the space of
all continuous,  real-valued functions on $X$ with compact support.
We denote the Borel $\sigma$-algebra on $\Gamma$ by
$\mathcal{B}(\Gamma)$. A point process in $X$ is a probability
measure on $(\Gamma, \mathcal{B}(\Gamma))$.

We fix a point process $\mu$ which satisfies the so-called condition
$(\Sigma'_{\nu})$ \cite{DVJ,MWM}, i.e., there exist a measurable function
$r: X\times \Gamma\rightarrow [0, +\infty]$, called the Papangelou
intensity of $\mu$, such that
\begin{equation}\label{3}
\int_{\Gamma}\mu(d\gamma)\int_{X}\gamma(dx)F(x,
\gamma)=\int_{\Gamma}\mu(d\gamma)\int_{X}\nu(dx)\,r(x,\gamma)F(x,\gamma\cup
x)
\end{equation}
for any measurable function $F: X\times \Gamma\rightarrow
[0,+\infty]$. The condition $(\Sigma'_{\nu})$ can be thought of as a kind
of weak Gibbsianess of $\mu$. Intuitively, we may treat the
Papangelou intensity as
\begin{equation}\label{7}
r(x,\gamma)=\exp[-E(x,\gamma)],
\end{equation}
where $E(x,\gamma)$ is the relative energy of interaction between
particle $x$ and configuration~$\gamma$.

To define an equilibrium Glauber dynamics for which $\mu$ is a
symmetrizing measure, we fix a death coefficient as a measurable
function $d:X\times \Gamma\rightarrow [0, +\infty]$, and then
define a birth coefficient $b: X\times\Gamma\rightarrow [0,
+\infty]$ by
\begin{equation}\label{5}
b(x,\gamma)=d(x,\gamma)r(x,\gamma), \quad (x,\gamma)\in X\times
\Gamma .
\end{equation}
To define a Kawasaki dynamics, we fix a measurable function
$c:X^{2}\times \Gamma^2\rightarrow [0,+\infty]$ which satisfies
\begin{equation}\label{4}
r(x,\gamma)c(x,y,\gamma)=r(y,\gamma)c(y,x,\gamma), \quad
(x,y,\gamma)\in X^{2}\times \Gamma.
\end{equation}
Formulas \eqref{5} and \eqref{4} are called the balance
conditions \cite{G1,G2}. We will also assume that the function $c(x,y,\gamma)$ vanishes if at
least one of the functions $r(x,\gamma)$ and $r(y,\gamma)$ vanishes, i.e.,
\begin{align}\label{6}
c(x,y,\gamma)=c(x,y,\gamma)\chi_{\{r>0\}}(x,\gamma)\chi_{\{r>0\}}(y,\gamma).
\end{align}
Here, for a set $A$, $\chi_{A}$ denotes the indicator function of
$A$. We refer to \cite[Remark 3.1]{LO} for a justification of this
assumption, which involves the interpretation of $r(x,\gamma)$ as in
\eqref{7}, see also Remark \ref{ghdstsgcf} below.

We denote by
$\mathcal{F}C_{\mathrm b}(C_{0}(X),\Gamma)$ the space of all
functions of the form
\begin{equation}\label{hdtrse}\Gamma\ni\gamma\mapsto F(\gamma)=g(\langle \varphi_{1},\gamma\rangle,\dots,\langle \varphi_{N},\gamma\rangle),\end{equation}
where $N\in \mathbb{N}$, $\varphi_{1},\dots, \varphi_{N}\in C_{0}(X)$
and $g\in C_{\mathrm b}(\mathbb{R}^{N})$.
Here, $C_{\mathrm b}(\mathbb{R}^{N})$ denotes the space of all continuous
bounded functions on $\mathbb{R}^{N}$.
We
assume that, for each $\Lambda\in \mathcal{B}_{0}(X)$,
\begin{gather}\label{8}
\int_{\Gamma}\mu(d\gamma)\int_{\Lambda}\gamma(dx)\,d(x,\gamma\setminus
x)<\infty,\\
\label{9}
\int_{\Gamma}\mu(d\gamma)\int_{X}\gamma(dx)\int_{X}\nu(dy)\,c(x,y,\gamma\setminus
x)(\chi_{\Lambda}(x)+\chi_{\Lambda}(y))<\infty.
\end{gather}
As easily seen, conditions \eqref{8} and \eqref{9} are sufficient in
order to define bilinear forms
\begin{align*}
{\mathcal{E}}_{\mathrm G}(F,G):&=\int_{\Gamma}\mu(d\gamma)\int_{X}\gamma(dx)\,d(x,\gamma\setminus
x)(F(\gamma\setminus x)-F(\gamma))(G(\gamma\setminus x)-G(\gamma)),
\\
{\mathcal{E}}_{\mathrm K}(F,G):&=\frac{1}{2}\int_{\Gamma}\mu(d\gamma)\int_{X}\gamma(dx)\int_{X}\nu(dy)\,c(x,y,\gamma\setminus
x)(F(\gamma\setminus x\cup y)-F(\gamma))
\\
&
\times
(G(\gamma\setminus x\cup y)-G(\gamma)),
\end{align*}
where $F,G\in \mathcal{F}C_{\mathrm b}(C_{0}(X),\Gamma)$.

For the construction of the Kawasaki dynamics, we will also assume
that the following technical assumptions holds:
\begin{equation}\label{10}
\begin{aligned}
& \exists u,v\in \mathbb{R}\quad \forall \Lambda\in\mathcal B_0(X):
\\
& \qquad  \int_{\Lambda}\gamma(dx)\int_{\Lambda}\nu(dy)\,r(x,\gamma\setminus
x)^{u}r(y,\gamma\setminus x)^{v}c(x,y,\gamma\setminus y)\in
L^{2}(\Gamma, \mu)<\infty.
\end{aligned}
\end{equation}
 Note that in formula
\eqref{10} and below, we use the convention $\frac{0}{0}:=0$.

The following theorem was essentially proved in \cite{LO}.

\begin{theorem}\label{asd}
 (i) Assume that a point process $\mu$ satisfies \eqref{3}. Assume that
conditions \eqref{5}, \eqref{8}, respectively
\eqref{4}, \eqref{6},
\eqref{9},  and \eqref{10} are satisfied. Let $\sharp=\mathrm G, \mathrm K$.
Then the bilinear form $(\mathcal{E}_{\sharp},
\mathcal{F}C_{\mathrm b}(C_{0}(x),\Gamma))$ is closable in $L^{2}(\Gamma,
\mu)$ and its closure will be denoted by $(\mathcal{E}_{\sharp},
D(\mathcal{E}_{\sharp}))$.
Further there exists a conservative Hunt process (Glauber, respectively
Kawasaki dynamics)
\begin{align*}
M^{\sharp}=\left(\Omega^{\sharp},
\mathcal{F}^{\sharp},(\mathcal{F}^{\sharp}_{t})_{t\geq
0},(\Theta_{t}^{\sharp})_{t\geq 0}, (X^{\sharp}(t))_{t\geq 0},
(P^{\sharp}_{\gamma})_{\gamma\in \Gamma}\right)
\end{align*}
on $\Gamma$ which is properly associated with
$(\mathcal{E}_{\sharp}, D(\mathcal{E}_{\sharp}))$, i.e., for all
($\mu$-version of) $F\in L^{2}(\Gamma, \mu)$ and $t>0$ $$\Gamma\ni
\gamma\mapsto
p^{\sharp}_{t}F(\gamma):=\int_{\Omega^\sharp}F(X^{\sharp}(t))\, dP_{\gamma}^{\sharp}$$
is an $\mathcal{E}^{\sharp}$-quasi continuous version of
$\exp(tL_{\sharp})F$, where $(-L_{\sharp}, D(L_{\sharp}))$ is the
generator of $(\mathcal{E}_{\sharp}, D(\mathcal{E}_{\sharp}))$.
$M^{\sharp}$ is up-to $\mu$-equivalence unique. In particular,
$M^{\sharp}$ is $\mu$-symmetric and has $\mu$ as invariant measure.

 (ii)  $M^\sharp$ from (i) is  up to $\mu$-equivalence unique
between all Hunt processes $$M'=\left(\Omega',
\mathcal{F}',(\mathcal{F}'_{t})_{t\geq 0}, (\Theta'_{t})_{t\geq 0},
(X'(t))_{t\geq 0}, (P'_{\gamma})_{\gamma\in \Gamma}\right)$$ on
$\Gamma$ having $\mu$ as invariant measure and solving a martingale
problem for $(L_{\sharp}, D(L_{\sharp}))$, i.e.,  for all $G\in
D(H_\sharp)$
$$\widetilde G({X}'(t))-\widetilde G({X}'(0))-\int_0^t (L_\sharp
G)({X}'(s))\,ds,\quad t\ge0,$$ is an $({\mathcal F}_t')$-martingale
under ${P}_\gamma'$ for ${\mathcal E}_\sharp$-q.e.\ $\gamma\in\Gamma$.
Here, $\widetilde G$ denotes an ${\mathcal E}_\sharp$-quasi-continuous
version of $G$.

 (iii) Further assume that, for each $\Lambda\in \mathcal{B}_{0}(X)$,
\begin{equation}\label{12}
\int_\Lambda \gamma(dx)\, d(x,\gamma\setminus x)\in L^2(\Gamma,\mu),\quad
\int_\Lambda \nu(dx)\,b(x,\gamma)\in L^2(\Gamma,\mu),
\end{equation}
in the Glauber case, and
\begin{equation}\label{13}
\int_X\gamma(dx)\int _X \nu(dy)\, c(x,y,\gamma\setminus
x)(\chi_\Lambda(x)+\chi_\Lambda(y))\in L^2(\Gamma,\mu)
\end{equation}
in the Kawasaki case. Then $\mathcal{F}C_{\mathrm b}(C_{0}(X),\Gamma)\subset D(L_\sharp)$,
and for each $F\in \mathcal{F}C_{\mathrm b}(C_{0}(X),\Gamma)$, $L_{\sharp}F$
is given by formulas \eqref{1} and \eqref{2}, respectively.
\end{theorem}

\begin{remark}
We refer to \cite{MR} for an explanation of notions appearing in
Theorem~\ref{asd}, see also a brief explanation of them in \cite{LO}.
\end{remark}

\begin{proof}[Proof of Theorem \ref{asd}]
 The statement follows from Theorems 3.1 and
3.2 in \cite{LO}. Note that, although these theorems are formulated
for determinantal point processes only, their proof only uses the
$(\Sigma'_{\nu})$ property of these point processes. Note also that
condition \eqref{10} is formulated in \cite{LO} only for $v=1$,
however the proof of Lemma~3.2 in \cite{LO} admits a straightforward generalization to the case of an arbitrary
$v\in \mathbb{R}$.
\end{proof}

\begin{remark}
Part (iii) of Theorem~\ref{asd} states that the operator
$(-L_{\sharp}, D(L_{\sharp}))$ is the Friedrichs' extention of the
operator $(-L_{\sharp}, \mathcal{F}C_{\mathrm b}(C_{0}(X), \Gamma))$ defined
by formulas \eqref{1}, \eqref{2}, respectively.
\end{remark}

Let us fix a parameter $s\in [0,1]$ and define
\begin{align}\label{14}
d(x,\gamma):&=r(x,\gamma)^{s-1}\chi_{\{r>0\}}(x,\gamma),\quad (x,
\gamma)\in X\times \Gamma,\\
\label{15}
b(x,\gamma):&=r(x,\gamma)^{s}\chi_{\{r>0\}}(x,\gamma), \quad (x,
\gamma)\in X\times \Gamma,
\end{align}
\begin{equation}\label{16}
\begin{gathered}
c(x,y,\gamma):=a(x,y)r(x,\gamma)^{s-1}r(y,\gamma)^{s}\chi_{\{r>0\}}(x,\gamma)\chi_{\{r>0\}}(y,\gamma),
\\
(x,y,\gamma)\in{X^{2}\times \Gamma}.
\end{gathered}
\end{equation}
Here the function $a:X^{2}\rightarrow [0,+\infty)$ is bounded,
measurable, symmetric (i.e., $a(x,y)=a(y,x)$), and satisfies
\begin{equation}\label{17}
\sup_{x\in X}\int_{X}a(x, y)\,\nu(dy)<\infty.
\end{equation}
Note that the balance conditions \eqref{5} and \eqref{4} are
satisfied for these coefficients, and so is condition \eqref{6}.

\begin{remark}\label{tesa4wa}
Note that, if $X=\mathbb R^d$ and $a(x,y)$ has the form $a(x-y)$ for a function $a:\mathbb R^d\to[0,\infty)$, then condition \eqref{17} means that $a\in L^1(\mathbb R^d,dx)$. (Here and below, in the case $X=\mathbb R^d$, we use an obvious abuse of notation.)
\end{remark}

\begin{remark}\label{ghdstsgcf}
Using representation \eqref{7}, we can rewrite formulas \eqref{14}--\eqref{16}  as follows:
\begin{align*}
d(x,\gamma\setminus x)&=\exp[(1-s)E(x, \gamma\setminus
x)]\chi_{\{E<+\infty\}}(x,\gamma\setminus x),\\
b(x,\gamma\setminus x)&=\exp[-sE(x,\gamma\setminus
x)]\chi_{\{E<+\infty\}}(x,\gamma\setminus x),\\
c(x,y,\gamma\setminus x)&=a(x,y)\exp[(1-s)E(x,\gamma\setminus
x)-sE(y,\gamma\setminus x)]\\
& \times\chi_{\{E<+\infty\}}(x,\gamma\setminus x)\chi_{\{E<+\infty\}}(y, \gamma\setminus x).
\end{align*}
So, if the corresponding dynamics exist, one can give the following
heuristic description  of them: Both dynamics are concentrated on
configurations $\gamma\in \Gamma$ such that, for each $x\in
\gamma$, the relative energy of interaction between $x$ and the rest
of configuration, $\gamma\setminus x$, is finite; those particles
tend to die, respectively hop, which have a high energy of
interaction with the rest of the configuration, while it is more
probable that a new particle is born at $y$, respectively $x$ hops
to $y$, if the energy of interaction between $y$ and the rest of the
configuration is low. \end{remark}

Let us assume that the point process $\mu$ satisfies:
$$\forall\Lambda\in \mathcal{B}_{0}(X):\quad
\int_{\Lambda}\gamma(dx)\in{L^{2}(\Gamma, \mu)}.$$ Then, by choosing
$u=1-s$ and $v=-s$ in \eqref{10}, we conclude that the coefficient
$c$ given by \eqref{16} satisfies \eqref{10}.

We will construct below a class of point processes
$\mu$ for which the coefficients $d$, $b$ and $c$ given above
satisfy the other conditions of Theorem \ref{asd}.

\section{Permanental point processes and corresponding equilibrium dynamics}
Let $K$ be a linear, bounded, self-adjoint operator on the real space
$L^{2}(X,\nu)$. Further assume that $K\geq 0$ and
$K$ is locally of trace class, i.e.,
$\operatorname{Tr}(P_{\Lambda}KP_{\Lambda})<\infty$ for all $\Lambda\in
\mathcal{B}_{0}(X)$, where $P_{\Lambda}$ denotes the operator of
multiplication by $\chi_{\Lambda}$. Hence, each operator
$P_{\Lambda}\sqrt{K}$ is of Hilbert--Schmidt class. Following \cite{LM}
(see also \cite[Lemma A.4]{GY}), we conclude that $\sqrt{K}$ is
an integral operator whose integral kernel, $\varkappa(x,y)$, satisfies
\begin{equation}\label{31}
\int_{\Lambda}\int_X\nu(dx)\nu(dy)\varkappa(x,y)^{2}<\infty \quad \text{for all
}\quad \Lambda\in \mathcal{B}_{0}(X).
\end{equation}
In particular,
\begin{equation}\label{32}
\varkappa(x,\cdot)\in L^{2}(X,\nu) \quad \text{for $\nu$-a.a.\quad $x\in X$.}
\end{equation}
Hence, $K$ is an integral operator whose integral kernel can be chosen as
\begin{equation}\label{33}
\begin{aligned}
k(x,y) & =\int_{X}\varkappa(x,z)\varkappa(z,y)\nu(dz)
\\
& =\int_{X}\varkappa(x,z)\varkappa(y,z)\nu(dz)
=(\varkappa(x,\cdot),\varkappa(y,\cdot))_{L^2(X,\nu)}.
\end{aligned}
\end{equation}
We also have, for each $\Lambda\in \mathcal{B}_{0}(X)$,
\begin{equation}\label{34}
\begin{aligned}
\operatorname{Tr}(P_{\Lambda}KP_{\Lambda})&=\|\sqrt{K}P_{\Lambda}\|^{2}_{\mathrm{HS}}
\\
&=\int_{\Lambda}\nu(dx)\int_{X}\nu(dy)\varkappa(x,y)^{2}
=\int_{\Lambda}k(x,x)\,\nu(dx),
\end{aligned}
\end{equation}
where $\|\cdot\|_{\mathrm{HS}}$ denotes the Hilbert--Schmidt norm.

\begin{prop}\label{37}
There exists a random field $(Y(x))_{x\in X}$ on a probability space
$(\Omega, \mathcal{A}, P)$ such that the mapping
\begin{equation}\label{35}
X\times \Omega \ni (x,\omega)\mapsto Y(x, \omega)
\end{equation}
is measurable, and for $\nu$-a.a.\ $x\in X$, $Y(x)$ is a  Gaussian random variable with mean $0$ and such that  \begin{equation}\label{36}
\mathbb{E}\left(Y(x)Y(y)\right)=k(x,y) \quad \text{for
$\nu^{\otimes 2}$-a.a.\ $(x,y)\in X^{2}$ and $\nu$-a.a.\ $x=y\in X$}.
\end{equation}
\end{prop}

\begin{remark}\label{cfdxrbv}
The statement of Proposition \ref{37} is well-known if the integral
kernel of the operator $K$ admits a continuous version (see e.g.\
Theorem~1.8 and p.~456 in \cite{ST}).  In the latter case,
$(Y(x))_{x\in X}$ is a Gaussian random field and formula \eqref{36}
holds for all $(x,y)\in X^{2}$.
\end{remark}

\begin{proof}[Proof of Proposition \ref{37}] Consider  a standard triple of real Hilbert spaces $$H_{+}\subset
H_{0}=L^{2}(X,\nu)\subset H_{-}\,.$$ Here the Hilbert space $H_{+}$ is
densely and continuously embedded into $H_{0}$, the inclusion
operator $H_{+}\hookrightarrow H_{0}$ is of Hilbert--Schmidt class, and the
Hilbert space $H_{-}$ is the dual space of $H_{+}$ with respect to the
center space $H_{0}$ (see e.g.\ \cite{BK}).

Let $\mathbb{P}$ be the standard Gaussian measure on $H_{-}$, i.e.,
the probability measure on the Borel $\sigma$-algebra
$\mathcal{B}(H_{-})$ which has Fourier transform
$$\int_{H_{-}}e^{i\langle \omega,
f\rangle}\,\mathbb{P}(d\omega)=\exp\Big[-\frac{1}{2}\|f\|
^{2}_{H_{0}}\Big],\quad f\in H_{+}\,,$$ where $\langle \omega, f\rangle$
denotes the dual pairing between $\omega\in H_{-}$ and $f\in
H_{+}$\,.  Then the mapping $H_{+}\ni f\rightarrow
\langle \cdot, f\rangle$ can be extended by continuity to an
isometry
\begin{equation}\label{38}
I: H_{0}\rightarrow L^{2}(H_{-},\mathbb{P}).
\end{equation}
For any $f\in H_{0}$ we denote $\langle \cdot,
f\rangle:=If$. Thus, for each $f\in H_{0}$, $\langle
\cdot, f\rangle$ is a (complex) Gaussian random variable with mean $0$  and
for any $f,g\in H_{0}$
\begin{equation}\label{39}
\int_{H_{-}}\langle \omega,f\rangle\langle
\omega,g\rangle\,\mathbb P(d\omega)=(f,g)_{L^2(X,\nu)}.
\end{equation}
Thus, by \eqref{32}, we set for $\nu$-a.a. $x\in X$,
$\widetilde{Y}(x,\omega):=\langle\omega, k(x,\cdot)\rangle$. Hence
$\widetilde{Y}(x)$ is a Gaussian random variable and by \eqref{33}
 and \eqref{39}, \eqref{36} holds.

Hence, it remains to prove that there exists a random field
$Y=(Y(x))_{x\in X}$ for which the mapping \eqref{35} is measurable
and such that $Y(x,\omega)=\widetilde{Y}(x,\omega)$  for
$\nu\otimes\mathbb{P}$-a.a.\ $(x,\omega)$. To this end, we fix any
$\Lambda\in \mathcal{B}_{0}(X)$ and denote by $\mathcal{B}(\Lambda)$
the trace $\sigma$-algebra of $\mathcal{B}(X)$ on $\Lambda$. We
define a set $\mathcal{D}_{\Lambda}$ of the functions
$u:\Lambda\times X\rightarrow \mathbb{R}$ of the form
\begin{equation}\label{40}
u(x,y)=\sum^{n}_{i=1}\chi_{\Delta_{i}}(x)f_{i}(y),
\end{equation}
where $\Delta_{i}\in \mathcal{B}(\Lambda)$, $f_{i}\in
H_{+}$, $i=1,\dots, n$. Define a linear mapping
\begin{align}\label{41}
I_{\Lambda}:\mathcal{D}_{\Lambda}\rightarrow
L^{2}(\Lambda\times H_{-}, \nu\otimes \mathbb{P})
\end{align}
by setting, for each $u\in \mathcal{D}_{\Lambda}$ of the form
\eqref{40},
$$
(I_{\Lambda}u)(x,\omega)=\sum^{n}_{i=1}\chi_{\Delta_{i}}(x)\langle\omega,f_{i}\rangle,\quad
(x,\omega)\in \Lambda\times H_{-}\,.
$$
Clearly, $I_{\Lambda}$ can be extended to an isometry
$$
I_{\Lambda}: L^{2}(\Lambda\times X, \nu^{\otimes
2})\rightarrow L^{2}(\Lambda\times H_{-}, \nu\otimes
\mathbb{P}),$$ and we have $I_{\Lambda}=\mathbf 1_{\Lambda}\otimes I$, where
$\mathbf 1_{\Lambda}$ is the identity operator in
$L^{2}(\Lambda, \nu)$ and the operator $I$ is as in
\eqref{38}.

Fix any $u\in  L^{2}(\Lambda\times X, \nu^{\otimes
2})$. As easily seen, there exist a sequence
$(u_{n})_{n=1}^{\infty}\subset \mathcal{D}_{\Lambda}$ such that
$u_n\to u$ in $L^2(\Lambda\times X,
\nu^{\otimes2})$ and  for
$\nu$-a.a. $x\in \Lambda$, $u_{n}(x,\cdot)\rightarrow u(x,\cdot)$ in
$L^{2}(X, \nu)$
Hence, for
$\nu$-a.a.\ $x\in \Lambda$,
$I_{\Lambda}u_{n}(x,\cdot)\rightarrow I_{\Lambda}u(x,\cdot)$ in
$L^{2}(H_{-},\mathbb{P})$, which implies
\begin{equation}\label{gh}
(I_{\Lambda}u)(x,\omega)=\langle \omega, u(x,\cdot)\rangle
\quad\text{for  $\mathbb{P}$-a.a.\  $\omega\in H_{-}$\,.}
\end{equation}

Now, denote by $\varkappa_{\Lambda}$ the restriction of $\varkappa$ to the set
$\Lambda\times X$. For $\nu$-a.a.\ $x\in \Lambda$, we define
$Y_{\Lambda}(x):=(I_{\Lambda}\varkappa_{\Lambda})(x,\cdot)$. Hence, by
\eqref{gh}, for $\nu$-a.a.\ $x\in \Lambda$,
$Y_{\Lambda}(x)=\widetilde{Y}(x)$ $\mathbb{P}$-a.e. Finally, let
$(\Lambda_{n})_{n=1}^{\infty}\subset \mathcal{B}_{0}(X)$ be such
that $\Lambda_{n}\cap \Lambda_{m}=\emptyset$ if $n\neq m$ and
$\bigcup^{\infty}_{n=1}\Lambda_{n}=X$. Setting
$Y(x):=Y_{\Lambda_{n}}(x)$ for $\nu$-a.a.\ $x\in \Lambda_{n}, n\in
\mathbb{N}$, we conclude  the statement.\end{proof}

Let $Y$ be a random field as in Proposition \ref{37}. For each
$\Lambda\in \mathcal{B}_{0}(X)$, we have
\begin{align*}
\mathbb{E}\left(\int_{\Lambda}Y(x)^{2}\,\nu(dx)\right)&=\int_{\Lambda}\mathbb E(Y(x)^{2})\,\nu(dx)\\
&=\int_{\Lambda}\nu(dx)\int_{X}\nu(dy)\varkappa(x,y)^{2}<\infty.
\end{align*}
In particular, the function $Y(x)^{2}$ is locally $\nu$-integrable
$\mathbb P$-a.s. Let $l\in\mathbb{N}$ and let $(\Omega, \mathcal{A}, \mathbb P)$ be
a probability space on which $l$ independent copies $Y_{1},
Y_{2},\ldots, Y_{l}$ of a random field $Y$ as in Proposition
\ref{37} are defined. Denote by $\mu^{(l)}$ the Cox point process on
$X$ with random intensity $g^{(l)}(x)=\sum^{l}_{i=1}Y_{i}(x)^{2}$,
which is locally $\nu$-integrable $\mathbb{P}$-a.s. Thus,
$\mu^{(l)}$ is the probability measure on $(\Gamma,
\mathcal{B}(\Gamma))$ which satisfies
\begin{equation}\label{44}
\int_{\Gamma}\mu^{(l)}(d\gamma)F(\gamma)=\int_{\Omega}\mathbb{P}(d\omega)\int_{\Gamma}\pi_{{g}^{(l)}(x,\omega)\nu(dx)}(d\gamma)F(\gamma)
\end{equation}
for each measurable function $F:\Gamma\rightarrow [0,+\infty]$.
Here, for a locally $\nu$-integrable function $g: X\rightarrow
[0,+\infty)$, we denote by $\pi_{g(x)\nu(dx)}$ the Poisson point
process in $X$ with intensity measure $g(x)\nu(dx)$, see
e.g\ \cite{DVJ}. This is the unique point
process in $X$ which satisfies the Mecke identity
\begin{equation}\label{45}
\int_{\Gamma}\pi_{g(x)\nu(dx)}(d\gamma)\,\int_{X}\gamma(dx)F(x,\gamma)=\int_{\Gamma}\pi_{g(x)\nu(dx)}(d\gamma)\int_{X}\nu(dx)\,g(x)F(x,\gamma\cup
x)
\end{equation}
for each measurable $F: X\times\Gamma \rightarrow [0,+\infty]$. By
\eqref{44} and \eqref{45} (compare with e.g.\ \cite{MW}), for each
$l\in \mathbb{N}$, the point process $\mu^{(l)}$ satisfies condition
$(\Sigma'_{\nu})$ and its Papangelou intensity is given by
\begin{align}
r^{(l)}(x,\gamma)=\widetilde{\mathbb{E}}(g^{(l)}(x)\mid\mathcal{F})(\gamma)
=\widetilde{\mathbb{E}}\Big(\sum^{l}_{i=1}Y_{i}(x)^{2}\mid\mathcal{F}\Big)(\gamma).\label{46}
\end{align}
Here $\widetilde{\mathbb{E}}$ denotes the (conditional) expectation
with respect to the probability measure
\begin{align}\label{47}
\widetilde{\mathbb{P}}(d\omega,d\gamma)=\widetilde{\mathbb{P}}(d\omega)\,\pi_{g^{(l)}(x,\omega)\nu(dx)}(d\gamma)
\end{align}
on  $\Omega\times\Gamma$, while $\mathcal{F}$ denotes the
$\sigma$-algebra on $\Omega\times\Gamma$ generated by the mappings
$$\Omega\times\Gamma\ni(\omega,\gamma)\rightarrow F(\gamma)\in
\mathbb{R},$$ where $F: \Gamma\rightarrow \mathbb{R}$ is measurable.

Recall that a point process $\mu$ in $X$ is said to have correlation
functions if, for each $n\in\mathbb{N}$, there exist a
non-negative, measurable, symmetric function $k_{\mu}^{(n)}$ on
$X^{n}$ such that, for any measurable, symmetric function $f^{n}:
X^{n}\rightarrow [0,+\infty]$,
\begin{equation}\label{48}
\begin{aligned}
&\int_{\Gamma}\sum_{\{x_{1},\dots,x_{n}\}\subset
\gamma}f^{(n)}(x_{1},\dots,x_{n})\,\mu(d\gamma)
\\
&\quad=\frac{1}{n!}\int_{X^{n}}f^{(n)}(x_{1},\dots,x_{n})k_{\mu}^{(n)}(x_{1},\dots,x_{n})\nu(dx_{1})\dotsm\nu(dx_{n}).
\end{aligned}
\end{equation}
As well known (e.g.\ \cite{DVJ}), for a locally $\nu$-integrable
function $g: X\rightarrow [0,+\infty)$, the Poisson point process
$\pi_{g(x)\nu(dx)}$ has correlation functions
\begin{align}\label{49}
k^{(n)}_{\mu}(x_{1},\ldots,x_{n})=g(x_{1})\dotsm g(x_{n}).
\end{align}

Let us recall the notion of $\alpha$-permanent \cite{VJ}, called
$\alpha$-determinant in \cite{ST}. For a square matrix
$A=(a_{ij})^{n}_{i,j=1}$ and $\alpha\in\mathbb{R}$, we set
$$\operatorname{per}_\alpha A:=\sum_{\sigma\in S_{n}}\alpha^{n-m(\sigma)}\prod^{n}_{i=1}a_{i\sigma(i)},$$
where $S_{n}$ is the group of all permutations of $\{1, \dots, n\}$ and
$m(\sigma)$ denotes the number of cycles in $\sigma$. In particular,
$\operatorname{per}_{1}A$ is the usual permanent of $A$, while
$\operatorname{per}_{-1}A$ is the usual determinant of $A$.
Analogously to \cite[subsec.\ 6.4]{ST}, we conclude from \eqref{44},
\eqref{48} and \eqref{49} that the point process $\mu^{(l)}$ has
correlation functions
\begin{equation}\label{50}
k_{\mu^{(l)}}^{(n)}(x_{1}, \dots,
x_{n})=\operatorname{per}_{\frac{l}{2}}(lk(x_{i},x_{j}))^{n}_{i,j=1}\quad
\text{for $\nu^{\otimes n}$-a.a.\ $(x_{1}, \dots, x_{n})\in X^{n}$.}
\end{equation}
For $l=2$, the point process $\mu^{(2)}$ is often called a boson point process, see e.g.\ \cite{DVJ,LM}. Thus, we have
proved the following

\begin{prop}\label{as}
For each $l\in \mathbb{N}$, there exists a point process $\mu^{(l)}$
in $X$ whose correlation functions are given by \eqref{50}. The
$\mu^{(l)}$ satisfies condition $(\Sigma'_{\nu})$ and its Papangelou
intensity is given by \eqref{46}.
\end{prop}

\begin{remark}
Recall that in \cite{ST}, under the same assumptions on the operator
$K$, the existence of a point process with correlation functions
\eqref{50} was proved for even $l\in \mathbb{N}$, and for odd $l\in
\mathbb{N}$ the statement of Proposition~\ref{as} was proved under the
additional assumption of continuity of the integral kernel
$k(\cdot,\cdot)$.
\end{remark}

We will now prove that, for a point process $\mu^{(l)}$ as in
Proposition~\ref{as},  Glauber and Kawasaki dynamics with coefficients
\eqref{14}, \eqref{15} and \eqref{16}, respectively exist.

\begin{theorem}\label{asdfg}
(i) For each point process $\mu^{(l)}$ as in Proposition~\ref{as}, the
coefficients $d(x,\gamma)$ and $b(x,\gamma)$ defined by \eqref{14}
and \eqref{15}, satisfy conditions \eqref{5} and \eqref{8} and so
statements  (i) and (ii) of Theorem \ref{asd} hold, in particular, a
corresponding Glauber dynamics exists.

(ii) Assume additionally that $k(x,x)$ is bounded outside a set
$\Delta\in\mathcal{B}_0(X)$. Then for a point process $\mu^{(l)}$
as in Proposition~\ref{as}, the coefficient $c(x,y,\gamma)$ defined by \eqref{16}, satisfies
\eqref{4}, \eqref{6}, \eqref{9} and \eqref{10}, and so statements
(i) and (ii) of Theorem~\ref{asd} hold, in particular, a
corresponding Kawasaki dynamics exists.
\end{theorem}

\begin{proof} We start with the following

\begin{lem}\label{serts} For each $n\in\mathbb N$ and for $\nu$-a.a.\ $x\in X$
\begin{equation}\label{trstr}
\int_\Gamma r(x,\gamma)^n\,\mu(d\gamma)\le\frac{(2n)!}{2^n\,n!}\,k(x,x)^n.
\end{equation}
\end{lem}

\begin{proof}
Using Jensen's inequality for conditional expectation and the formula for moments of a Gaussian measure
(see e.g.\ \cite[Chapter~2, Section~2, Lemma~2.1]{BK}), we have
\begin{align*}
\int_\Gamma r(x,\gamma)^n\,\mu(d\gamma)&=
\widetilde{\mathbb E}(\widetilde{\mathbb E}(Y(x)^2\mid\mathcal F)^n)
\le\widetilde{\mathbb E}(\widetilde{\mathbb E}(
Y(x)^{2n}\mid\mathcal F))
\\
&=\widetilde{\mathbb E}(Y(x)^{2n})
\le\frac{(2n)!}{2^n\,n!}\,\|\varkappa(x,\cdot)\|^{2n}_{L^2(X,\nu)}
=\frac{(2n)!}{2^n\,n!}\,k(x,x)^n
\end{align*}
for $\nu$-a.a.\ $x\in X$. \end{proof}

We will only prove statement (ii) of
Theorem~\ref{asdfg}, as the proof of statement (i) is similar and
simper. Also, for simplicity of notation, we will only consider the
case $l=1$ (for $l>1$ the proof being similar). We will also
omit the upper index (1) from our notation. By \eqref{3} we have, for
each $\Lambda\in \mathcal{B}_{0}(X)$,
\begin{equation}\label{55}
\begin{aligned}
&\int_{\Gamma}\mu(d\gamma)\int_{X}\gamma(dx)\int_{X}\nu(dy)\,c(x,y,\gamma\setminus
x)(\chi_{\Lambda}(x)+\chi_{\Lambda}(y))
\\
&\quad=\int_{\Gamma}\mu(d\gamma)\int_{X}\nu(dx)\int_{X}\nu(dy)\,r(x,\gamma)c(x,y,\gamma)(\chi_{\Lambda}(x)
+\chi_{\Lambda}(y))
\\
&\quad=\int_{\Gamma}\mu(d\gamma)\int_{X}\nu(dx)\int_{X}\nu(dy)\,a(x,y)r(x,\gamma)^{s}r(y,\gamma)^{s}
\chi_{\{r>0\}}(x,\gamma)
\\
&\quad\times\chi_{\{r>0\}}(y,\gamma)(\chi_{\Lambda}(x)+\chi_{\Lambda}(y))
\\
& \quad \leq\int_{\Gamma}\mu(d\gamma)\int_{X}\nu(dx)\int_{X}\nu(dy)\,a(x,y)r(x,\gamma)^{s}r(y,\gamma)^{s}(\chi_{\Lambda}(x)
+\chi_{\Lambda}(y))
\\
& \quad =2\int_{\Gamma}\mu(d\gamma)\int_{\Lambda}\nu(dx)\int_{X}\nu(dy)\,a(x,y)r(x,\gamma)^{s}r(y,\gamma)^{s}
\\
&\quad \leq2\int_{\Gamma}\mu(d\gamma)\int_{\Lambda}\nu(dx)\int_{X}\nu(dy)\,a(x,y)(1+r(x,\gamma))(1+r(y,\gamma)).
\end{aligned}
\end{equation}
By \eqref{17}
\begin{align}\label{56}
\int_{\Gamma}\mu(d\gamma)\int_{\Lambda}\nu(dx)\int_{X}\nu(dy)\,a(x,y)<\infty.
\end{align}
Below, $C_{i}, i=1,2,3,\dots$,  will denote positive constants
whose explicit values are not important for us. We have, by \eqref{17}
\begin{equation}\label{57}
\begin{aligned}
&\int_{\Gamma}\mu(d\gamma)\int_{\Lambda}\nu(dx)\int_{X}\nu(dy)\,a(x,y)r(x,\gamma)
\\
&\quad=\int_{\Gamma}\mu(d\gamma)\int_{\Lambda}\nu(dx)\,r(x,\gamma)\Big(\int_{X}\nu(dy)\,a(x,y)\Big)
\\
&\quad\leq C_{1}\int_{\Gamma}\mu(d\gamma)\int_{\Lambda}\nu(dx)r(x,\gamma)
\\
&\quad=C_{1}\int_{\Gamma}\mu(d\gamma)\int_{\Lambda}\gamma(dx)=C_{1}\int_{\Lambda}k(x,x)\,\nu(dx)<\infty.
\end{aligned}
\end{equation}
Next, by \eqref{46}
\begin{equation}\label{58}
\begin{aligned}
&\int_{\Gamma}\mu(d\gamma)\int_{\Lambda}\nu(dx)\int_{X}\nu(dy)\,a(x,y)r(y,\gamma)
\\
&\quad =\int_{\Lambda}\nu(dx)\int_{X}\nu(dy)\,a(x,y)\int_{\Gamma}\mu(d\gamma)r(y,\gamma)
\\
&\quad =\int_{\Lambda}\nu(dx)\int_{X}\nu(dy)\,a(x,y)k(y,y)
\\
&\quad =\int_{\Lambda}\nu(dx)\int_{\Delta}\nu(dy)\,a(x,y)k(y,y)+
\int_{\Lambda}\nu(dx)\int_{\Delta^c}\nu(dy)\,a(x,y)k(y,y)
\\
&\quad\le C_2\int_{\Lambda}\nu(dx)\int_{\Delta}\nu(dy)k(y,y)+C_3\int_{\Lambda}\nu(dx)\int_{\Delta^c}\nu(dy)\,a(x,y)<\infty,
\end{aligned}
\end{equation}
where we used that the function $a$ is bounded and $k(y,y)$ is bounded
on ${\Delta}^{c}$.
Analogously,  using Lemma \ref{serts}, we have
\begin{equation}\label{rtser5s}
\begin{aligned}
&\int_{\Gamma}\mu(d\gamma)\int_{\Lambda}\nu(dx)\int_{X}\nu(dy)a(x,y)r(x,\gamma)r(y,\gamma)
\\
&\quad\leq\int_{\Lambda}\nu(dx)\int_{X}\nu(dy)\,a(x,y)\|r(x,\cdot)\|_{L^{2}(\mu)}\,\|r(y,\cdot)\|_{L^{2}(\mu)}
\\
&\quad\leq C_{4}\int_{\Lambda}\nu(dx)\int_{X}\nu(dy)\,a(x,y)k(x,x)k(y,y)
\\
&\quad\leq
C_{5}\int_{\Lambda}\nu(dx)\,k(x,x)\int_{\Delta}\nu(dy)\,k(y,y)
\\
&\quad+C_{6}\int_{\Lambda}\nu(dx)k(x,x)\int_{{\Delta}^{c}}\nu(dy)\,a(x,y)<\infty.
\end{aligned}
\end{equation}
Thus, by \eqref{55}--\eqref{rtser5s}, the theorem is proven.
\end{proof}

\begin{theorem}\label{sdisd} (i) Let $s\in\left[\frac12,1\right]$, and let the conditions of Theorem~\ref{asdfg} (i) be satisfied. Then the
coefficients $d(x,\gamma)$ and $b(x,\gamma)$ defined by \eqref{14}
and \eqref{15}, satisfy condition \eqref{12}. Thus,
 $\mathcal{F}C_{\mathrm b}(C_{0}(X),\Gamma)\subset D(L_{\mathrm G})$,
and for each $F\in \mathcal{F}C_{\mathrm b}(C_{0}(X),\Gamma)$, $L_{\mathrm G}F$
is given by formula \eqref{1}.

(ii) Let $s\in\left[\frac12,1\right]$, and let the conditions of Theorem~\ref{asdfg} (ii) be satisfied.
Further assume that either
\begin{equation}\label{fdygsdet}
\forall\Lambda\in\mathcal B_0(X)\ \exists\Lambda'\in\mathcal B_0(X)\ \forall x\in\Lambda\ \forall y\in(\Lambda')^c:\quad a(x,y)=0,\end{equation}
or
\begin{equation}\label{fdre65}
\int_\Delta k(x,x)^2\,\nu(dx)<\infty,
\end{equation}
where $\Delta$ is as in Theorem~\ref{asdfg} (ii).
 Then the
coefficient $c(x,y,\gamma)$ defined by \eqref{16}, satisfies condition \eqref{13}. Thus,
 $\mathcal{F}C_{\mathrm b}(C_{0}(X),\Gamma)\subset D(L_{\mathrm K})$,
and for each $F\in \mathcal{F}C_{\mathrm b}(C_{0}(X),\Gamma)$, $L_{\mathrm K}F$
is given by formula \eqref{2}.
\end{theorem}

\begin{remark} If $X=\mathbb R^d$ and the function $a$ is as in Remark~\ref{tesa4wa}, then condition \eqref{fdygsdet}
   means that the function $\tilde a$ has a compact support.
   \end{remark}
\begin{proof}[Proof of Theorem \ref{sdisd}]
We again prove only the part related to Kawasaki dynamics and only in the case $l=1$, omitting the upper index (1) from our notation. We first
assume that \eqref{fdygsdet} is satisfied. Since the function $a$ is
bounded and satisfies \eqref{fdygsdet}, it suffices to show that,
for each $\Lambda\in \mathcal{B}_{0}(X)$,
\begin{equation}\label{asgasdg}
\int_{\Lambda}\gamma(dx)\int_{\Lambda}\nu(dy)r(x,\gamma\setminus
x)^{s-1}r(y,\gamma\setminus x)^{s}\chi_{\{r>0\}}(x,\gamma\setminus
x)\chi_{\{r>0\}}(y,\gamma\setminus x)\in L^{2}(\mu).
\end{equation}
We note that, for $s\in\left [\frac{1}{2},1\right]$, $2s-1\in [0,1]$.
Therefore, by the Cauchy inequality, we have
\begin{align}
&\int_{\Gamma}\mu(d\gamma)\Big(\int_{\Lambda}\gamma(dx)\,r(x,\gamma\setminus
x)^{s-1}\chi_{\{r>0\}}(x,\gamma\setminus
x)\nonumber
\\
& \quad \times \int_{\Lambda}\nu(dy)\,r(y,\gamma\setminus
x)^{s}\chi_{\{r>0\}}(y,\gamma\setminus x)\Big)^{2}\nonumber
\\
&\quad\leq\int_{\Gamma}\mu(d\gamma)\int_{\Lambda}\gamma(dx)\,r(x,\gamma\setminus
x)^{2(s-1)}\chi_{\{r>0\}}(x,\gamma\setminus
x)\notag\\
&\quad\times\Big(\int_{\Lambda}\nu(dy)\,r(y,\gamma\setminus
x)^{s}\chi_{\{r>0\}}(y,\gamma\setminus x)\Big)^{2} \gamma(\Lambda)\nonumber
\\
&\quad=\int_{\Gamma}\mu(d\gamma)\int_{\Lambda}\nu(dx)\,r(x,\gamma)^{2s-1}\chi_{\{r>0\}}(x,\gamma)\notag
\\
&\quad\times
\Big(\int_{\Lambda}\nu(dy)r(y,\gamma)^{s}\chi_{\{r>0\}}(y,\gamma)\Big)^{2}(\gamma(\Lambda)+1)\nonumber
\\
&\quad\leq
\int_{\Gamma}\mu(d\gamma)\Big(\int_{\Lambda}\nu(dx)(1+r(x,\gamma))\Big)^{3}(\gamma(\Lambda)+1)\nonumber
\\
&\quad\leq\Big(\int_{\Gamma}\mu(d\gamma)\Big(\int_{\Lambda}\nu(dx)(1+r(x,\gamma))\Big)^{6}\Big)^{1/2}
\Big(\int_{\Gamma}\mu(d\gamma)(\gamma(\Lambda)+1)^{2}\Big)^{1/2}.\label{ghsghsdf}
\end{align}

By  Lemma \ref{serts}, we have, for each $n\in\mathbb{N}$,
\begin{equation}\label{hjkhjfd}
\begin{aligned}
&\int_{\Gamma}\mu(d\gamma)\Big(\int_{\Lambda}\nu(dx)\,r(x,\gamma)\Big)^{n}
\\
&\quad=\int_{\Lambda}\nu(dx_{1})\dotsm\int_{\Lambda}\nu(dx_{n})\int_{\Gamma}\mu(d\gamma)\,
r(x_1,\gamma)\dotsm
r(x_{n},\gamma)\
\\
&\quad\leq\int_{\Lambda}\nu(dx_{1})\dotsm\int_{\Lambda}\nu(dx_{n})\|r(x_{1},\cdot)\|_{L^{n}(\mu)}\dotsm\|r(x_{n},\cdot)
\|_{L^{n}(\mu)}
\\
&\quad\leq\frac{(2n)!}{2^{n}n!}\Big(\int_{\Lambda}\nu(dx)k(x,x)\Big)^{n}<\infty
\end{aligned}
\end{equation}
Now, \eqref{asgasdg} follows from \eqref{ghsghsdf} and \eqref{hjkhjfd}.

Next, we assume that \eqref{fdre65} is satisfied. We fix
$\Lambda\in\mathcal{B}_{0}(X)$ and denote
$$u(x,y):=a(x,y)(\chi_{\Lambda}(x)+\chi_{\Lambda}(x)).$$
Then, by the Cauchy inequality,
\begin{align*}
&\int_{\Gamma}\mu(d\gamma)\Big(\int_{X}\gamma(dx)\int_{X}\nu(dy)\,u(x,y)r(x,\gamma\setminus
x)^{s-1}\chi_{\{r>0\}}(x,\gamma\setminus x)
\\
&\quad
\times r(y,\gamma\setminus
x)^{s}\chi_{\{r>0\}}(y,\gamma\setminus x)\Big)^{2}
\\
&\quad\leq
\int_{\Gamma}\mu(d\gamma)\int_{X}\gamma(dx)\int_{X}\nu(dy)\,u(x,y)r(x,\gamma\setminus
x)^{2(s-1)}\chi_{\{r>0\}}(x,\gamma\setminus x)
\\
&\quad\times r(y,\gamma\setminus
x)^{2s} \chi_{\{r>0\}}(y,\gamma\setminus
x)\int_{X}\gamma(dx')\int_{X}\nu(dy')\,u(x',y')
\\
&\quad =\int_{\Gamma}\mu(d\gamma)\int_{X}\nu(dx)\int_{X}\nu(dy)\,
u(x,y)r(x,\gamma)^{2s-1}\chi_{\{r>0\}}(x,\gamma)
\\
&\quad\times r(y,\gamma)^{2s}\chi_{\{r>0\}}
(y,\gamma)\int_{X}(\gamma+\varepsilon_{x})(dx')\int_{X}\nu(dy')\,u(x',y')
\\
&\quad\leq
\int_{\Gamma}\mu(d\gamma)\int_{X}\nu(dx)\int_{X}\nu(dy)\,u(x,y)(1+r(x,\gamma))(1+r(y,\gamma)^{2}
\\
&\quad\times\Big(\int_{X}\gamma(dx')\int_{X}\nu(dy')u(x',y')+\int_{X}\nu(dy')u(x,y')\Big).
\end{align*}
By \eqref{17},  it suffices to prove
that
\begin{gather}\label{aisv}
\int_{\Gamma}\mu(d\gamma)\Big(\int_{X}\nu(dx)\int_{X}\nu(dy)\,u(x,y)(1+r(x,\gamma))
(1+r(y,\gamma)^{2})\Big)^{2}<\infty,\\
\label{dfgsdfh}
\int_{\Gamma}\mu(d\gamma)\Big(\int_{X}\gamma(dx)\int_{X}\nu(dy)\,u(x,y)\Big)^{2}<\infty.
\end{gather}
We first to prove \eqref{dfgsdfh}. We have, by Proposition \ref{as},
\begin{align*}
&\int_{\Gamma}\Big(\int_{X}\gamma(dx)\int_{X}\nu(dy)u(x,y)\Big)^{2}
\\
&\quad =\int_{X}\nu(dy)\int_{X}\nu(dy')\int_{\Gamma}\mu(d\gamma)\int_{X}\gamma(dx)\int_{X}\gamma(dx')
\,u(x,y)u(x',y')
\\
&\quad=\int_{X}\nu(dy)\int_{X}\nu(dy')\int_{\Gamma}\mu(d\gamma)\Big(\int_{X}\gamma(dx)\,u(x,y)u(x,y')
\\
&\quad
+\int_{X}\gamma(dx)\int_{X}(\gamma-\varepsilon_{x})(dx')\,u(x,y)u(x',y')\Big)
\\
&\quad=\int_{X}\nu(dy)\int_{X}\nu(dy')\Big(\int_{X}\nu(dx)\,k(x,x)u(x,y)u(x,y')
\\
& \quad
+\int_{X}\nu(dx)\int_{X}\nu(dx')\Big(\frac12\,k(x,x')^{2}+k(x,x)k(x',x')\Big)u(x,y)u(x',y')\Big)
\\
&\quad\leq\int_{X}\nu(dy)\int_{X}\nu(dy')\Big(\int_{X}\nu(dx)\,k(x,x)u(x,y)u(x,y')
\\
&\quad
+\int_{X}\nu(dx)\int_{X}\nu(dx')\,\frac32\,k(x,x)k(x',x')u(x,y)u(x',y')\Big)
\\
&\quad=\int_{X}\nu(dy)\int_{X}\nu(dy')\int_{X}\nu(dx)\,k(x,x)u(x,y)u(x,y')
\\
&\quad
+\frac32\Big(\int_{X}\nu(dy)\int_{X}\nu(dx)\,k(x,x)u(x,y)\Big)^{2}
\\
&\quad\leq\int_{\Delta}\nu(dx)\,k(x,x)\Big(\int_{X}\nu(dy)\,u(x,y)\Big)^{2}
\\
&\quad\text{}+C_{7}\int_{X}\nu(dy)\int_{X}\nu(dy')\int_{X}\nu(dx)\,u(x,y)u(x,y')
\\
&\quad\text{}+\frac32\Big(\int_{\Delta}\nu(dx)\,k(x,x)\int_{X}\nu(dy)\,u(x,y)+C_{7}\int_{X}\nu(dy)
\int_{X}\nu(dx)\,u(x,y)
\Big)^{2}<\infty.
\end{align*}
Next, we prove \eqref{aisv}. By Lemma \ref{serts} and \eqref{fdre65},
we have
\begin{align*}
&\int_{\Gamma}\mu(d\gamma)\Big(\int_{X}\nu(dx)\int_{X}\nu(dy)\,u(x,y)(1+r(x,\gamma))(1+r(y,\gamma)^{2})
\Big)^{2}
\\
&\quad=\int_{X}\nu(dx)\int_{X}\nu(dx')\int_{X}\nu(dy)\int_{X}\nu(dy')\,u(x,y)u(x',y')
\\
&\quad\times\int_{\Gamma}\mu(d\gamma)(1+r(x,\gamma))(1+r(x',\gamma))(1+r(y,\gamma)^{2})(1+r(y',\gamma)^{2})
\\
&\quad\leq\int_{X}\nu(dx)\int_{X}\nu(dx')\int_{X}\nu(dy)\int_{X}\nu(dy')\,u(x,y)u(x',y')\left(1+\|r(x,\cdot)
\|_{L^{4}(\mu)}\right)
\\
&\quad\times\left(1+\|r(x',\cdot)\|_{L^{4}(\mu)}\right)\left(1+\|r(y,\cdot)^2\|_{L^{4}(\mu)}\right)
\left(1+\|r(y',\cdot)^2\|_{L^{4}(\mu)}\right)\\
&\quad\leq
C_{8}\Big(\int_{X}\nu(dx)\int_{X}\nu(dy)\,u(x,y)(1+k(x,x))(1+k(y,y)^{2})\Big)^{2}<\infty.
\end{align*}
Thus, the theorem is proven.\end{proof}

\section{Diffusion approximation}

From now on, we set $X=\mathbb R^d$, $d\in\mathbb N$, and $\nu$ to be Lebesgue measure. We will show that, under an appropriate scaling, the Dirichlet form of the Kawasaki dynamics converges to a Dirichlet form which identifies a diffusion process on $\Gamma$ having a permanental point process $\mu^{(l)}$ as a symmetrizing measure. The way we  scale the Kawasaki dynamics will be similar to the ansatz of \cite{KKL}.

We denote by
$\mathcal{F}C^\infty_{\mathrm b}(C^\infty_{0}(\mathbb R^d),\Gamma)$ the space of all
functions of the form \eqref{hdtrse} where $N\in \mathbb{N}$, $\varphi_{1},\dots, \varphi_{N}\in C^\infty_{0}(\mathbb R^d)$
and $g\in C^\infty_{\mathrm b}(\mathbb{R}^{N})$.
Here, $C^\infty_{0}(\mathbb R^d)$ denotes the space of smooth functions on $\mathbb R^d$ with compact support, and $C^\infty_{\mathrm b}(\mathbb{R}^{N})$ denotes the space of all smooth
bounded functions on $\mathbb{R}^{N}$ whose all derivatives are bounded. Clearly,
$$\mathcal{F}C^\infty_{\mathrm b}(C^\infty_{0}(\mathbb R^d),\Gamma)\subset \mathcal{F}C_{\mathrm b}(C_{0}(\mathbb R^d),\Gamma),$$
and  the set $\mathcal{F}C^\infty_{\mathrm b}(C^\infty_{0}(\mathbb R^d),\Gamma)$ is a core for the Dirichlet form $(\mathcal E_{\mathrm K}, D(\mathcal E_{\mathrm K}))$.

We fix $s=1/2$. Let us assume that the function $a(x,y)$
is as in Remark \ref{tesa4wa}. Thus, the coefficient $c(x,y,\gamma)$ has the form
\begin{equation}\label{sersa} c(x,y,\gamma)=a(x-y)r(x,\gamma)^{-1/2}r(y,\gamma)^{1/2}\chi_{\{r>0\}}(x,\gamma)\chi_{\{r>0\}}(y,\gamma)
.\end{equation}
Note that $y-x$ describes the change of the position of a particle which hops from $x$ to $y$. We now scale the function $a$ as follows: for each $\varepsilon>0$, we denote
\begin{equation}\label{drh}a_\varepsilon(x):=\varepsilon^{-d-2}a(x/\varepsilon),\quad x\in\mathbb R^d.\end{equation}
The Dirichlet form $(\mathcal E_{\mathrm K}, D(\mathcal E_{\mathrm K}))$ which corresponds to
the choice of function $a$ as in \eqref{drh} will be denoted by $(\mathcal E_\eps,D(\mathcal E_\eps))$.

\begin{theorem}\label{trsres} Assume that the function $a$ has  compact support, and the value $a(x)$ only depends on $|x|$, i.e., $a(x)=\tilde a(|x|)$ for some function $\tilde a:[0,\infty)\to\mathbb R$. Further assume that the function $\varkappa(x,y)$ has the form $\varkappa(x-y)$ for some $\varkappa:\mathbb R^d\to\mathbb C$, and
\begin{equation}\label{tsrs5r}
\lim_{y\to0}\int_{\mathbb R^d} (\varkappa(x)-\varkappa(x+y))^2\,dx=0.
\end{equation}
For each $l\in\mathbb N$, define a bilinear form $(\mathcal E_0,\mathcal{F}C^\infty_{\mathrm b}(C^\infty_{0}(\mathbb R^d),\Gamma))$ by
\begin{equation}\label{sawawe} \mathcal E_0(F,G):=c\int_\Gamma \mu^{(l)}(d\gamma)\int_{\mathbb R^d}dx\, r(x,\gamma)
\la \nabla_x F(\gamma\cup x),\nabla_x G(\gamma\cup x)\ra.\end{equation}
Here
$$c:=\frac{1}{2}\int_{\mathbb R^d}a(x)x_1^2\,dx$$
($x_1$ denoting the first coordinate of $x\in\mathbb R^d$), $\nabla_x$ denotes the gradient in the $x$ variable, and $\la\cdot,\cdot\ra$ stands for the scalar product in $\mathbb R^d$.  Then, for any $F,G\in \mathcal{F}C^\infty_{\mathrm b}(C^\infty_{0}(\mathbb R^d),\Gamma)$,
$$\mathcal E_\eps(F,G)\to\mathcal E_0(F,G)\quad \text{as }\eps\to0.$$
\end{theorem}

\begin{remark}\label{gdctsjh}
Assume that the function $\varkappa$ is differentiable on $\mathbb R^d$.
Denote
$$K(x,\delta):=\sup_{y\in B(x,\delta)}|\nabla \varkappa(y)|, \quad x\in\mathbb R^d,\quad \delta>0.$$
Here $B(x,\delta)$ denotes the closed ball in $\mathbb R^d$ centered at $x$ and of radius $\delta$. Assume that, for some $\delta>0$,
\begin{equation}\label{sraserw}
K(\cdot,\delta)\in L^2(\mathbb R^d,dx).
\end{equation}
Then condition \eqref{tsrs5r} is clearly satisfied.
Note that condition \eqref{sraserw} is slightly stronger than the condition $|\nabla\varkappa|\in L^2(\mathbb R^d,dx)$.
\end{remark}
\begin{proof}[Proof of Theorem \ref{trsres}] Again we will only present the proof in the case $l=1$, omitting the upper index $(1)$. We start with the following

\begin{lem}\label{ufytcfdf} Fix any $\Lambda\in\mathcal B_0(\mathbb R^d)$ and $\alpha\in(0,1]$.
Then, under the conditions of Theorem~\ref{trsres},
$$ r(x+\eps y,\gamma)^\alpha\to r(x,\gamma)^\alpha \quad \text{in}\quad  L^2(\Gamma\times\Lambda\times\mathbb R^d,\mu(d\gamma)
\,dx\,dy\, a(y))\quad \text{as}\quad \eps\to0.$$
\end{lem}

\begin{proof} We first prove the statement for $\alpha=1$. Thus, equivalently we have to prove that
 \begin{equation}\label{ctdstrmbgy}
 r(x+\eps y,\gamma)\to r(x,\gamma)\quad  \text{in} \quad L^2(\Omega\times\Gamma
 \times\Lambda\times\mathbb R^d,\tilde{\mathbb P}(d\omega,d\gamma)\,dx\,dy\, a(y))\quad \text{as}\quad \eps\to0.
 \end{equation}
We have, using Jensen's inequality for conditional expectation,
\begin{equation}\label{xsersaw}
\begin{aligned}
&\int_\Lambda dx\int_{\mathbb R^d}dy\,a(y)\int_{\Omega\times\Gamma}\tilde{\mathbb P}(d\omega,d\gamma)
\,(r(x+\eps y)-r(x,\gamma))^2
\\
&\quad =\int_\Lambda dx\int_{\mathbb R^d}dy\,a(y)\int_{\Omega\times\Gamma}\tilde{\mathbb P}(d\omega,d\gamma)
 \tilde{\mathbb E}(Y(x+\eps y)^2-Y(x)^2\mid\mathcal F)^2
 \\
&\quad \le\int_\Lambda dx\int_{\mathbb R^d}dy\,a(y)\int_{\Omega\times\Gamma}\tilde{\mathbb P}(d\omega,d\gamma)
 (Y(x+\eps y)^2-Y(x)^2)^2
 \\
&\quad =\int_\Lambda dx\int_{\mathbb R^d}dy\,a(y)\int_{\Omega}d{\mathbb P}\, (Y(x+\eps y)^4+Y(x)^4
-2Y(x+\eps y)^2Y(x)^2).
\end{aligned}
\end{equation}
Using the formula for moments of a Gaussian measure, we have
\begin{equation}\label{dserazaz}
\begin{aligned}
&\int_{\Omega} Y(x+\eps y)^4\,d{\mathbb P}
\\
&\quad=
3\Big(\int_{\mathbb R^d}
\varkappa(x+\eps y-u)^2\,du\Big)^2
\\
&\quad=3\Big(\int_{\mathbb R^d}
\varkappa(x-u)^2\,du\Big)^2
\\
&\quad=\int_{\Omega} Y(x)^4\,d{\mathbb P}.
\end{aligned}
\end{equation}
Analogously, using condition \eqref{tsrs5r} and the dominated convergence theorem, we get
\begin{equation}\label{drtdsrfts}
\begin{aligned}
&\int_\Lambda dx\int_{\mathbb R^d}dy\,a(y)\int_{\Omega}d{\mathbb P}\,Y(x+\eps y)^2Y(x)^2
\\
&\quad=\int_\Lambda dx\int_{\mathbb R^d}dy\,a(y)\bigg[
\int_{\mathbb R^d}\varkappa(x+\eps y-u)^2 \,du\cdot \int_{\mathbb R^d} \varkappa(x-u')^2\,du'
\\
&\quad\text{}+2\Big(\int_{\mathbb R^d}  \varkappa(x+\eps y-u)\varkappa(x-u)\,du\Big)^2\bigg]
\\
&\quad\to\int_\Lambda dx\int_{\mathbb R^d}dy\,a(y)\int_{\Omega}d{\mathbb P}\, Y(x)^4\quad \text{as}\quad \eps\to0.
\end{aligned}
\end{equation}
By \eqref{xsersaw}--\eqref{drtdsrfts}, statement \eqref{ctdstrmbgy} follows.

To prove the result for  $\alpha\in(0,1)$, it is now sufficient to show the following

{\it Claim.} Let $(\mathbf A,\mathcal A,m)$ be a measure space and let $m(A)<\infty$. Let
$f_\eps\in L^2(m)$,  $f_\eps\ge0$, $\eps\in[-1,1]$,  and let $f_\eps\to f_0$ in $L^2(m)$ as $\eps\to0$. Then, for each $\alpha\in(0,1)$, $f_\eps^\alpha\to f_0^\alpha$ in $L^2(m)$ as $\eps\to0$.

By e.g.\ \cite[Theorems 21.2 and 21.4]{Bauer}, $f_\eps\to f_0$ in $L^2(m)$  implies that
\begin{itemize}
\item[(i)] $f_\eps\to f_0$ in measure;
\item[(ii)] $\displaystyle\sup_{\eps\in[-1,1]} \int f_\eps^2\,dm<\infty$;
\item[(iii)] For each $\theta>0$ there exist $h\in L^1(m)$ and $\delta>0$ such that, for all $0<|\eps|\le1$ and
for each $A\in\mathcal A$
$$ \int_ A h\,dm\le \delta \Rightarrow \int _Af_\eps^2\,dm\le \theta.$$
\end{itemize}

Hence, for $\alpha\in(0,1)$, we get
\begin{itemize}
\item[a)] $f_\eps^\alpha\to f_0^\alpha$ in measure;
\item[b)] $\displaystyle\sup_{\eps\in[-1,1]} \int f_\eps^{2\alpha}\,dm\le \sup_{\eps\in[-1,1]} \int (1+f_\eps^2)\,dm<\infty;$
    \item[c)] Let $\theta$, $h$, and $\delta$ be as in (iii). Set $h':=h+\frac{\delta}{\theta}$. Clearly, $h\in L^1(m)$.
    Assume that, for some $A\in\mathcal A$, $\int_A h'\,dm\le\delta$. Hence
$\int_A h\,dm\le\delta$, and therefore $\int_A f_\eps^2\,dm\le\delta$ for all $0<|\eps|\le1$. Furthermore, we get
$\int_A\frac{\delta}{\theta}\,dm\le\delta$, and therefore $m(A)\le \theta$. Now
$$\int_A f_\eps^{2\alpha}\,dm\le \int_A(1+f_\eps^2)\,dm\le2\theta.$$
\end{itemize}
Applying again \cite[Theorems 21.2 and 21.4]{Bauer}, we conclude the claim.
\end{proof}

Fix any $F \in\mathcal{F}C^\infty_{\mathrm b}(C^\infty_{0}(\mathbb R^d),\Gamma)$. We have
\begin{align*}
&\mathcal E_\eps(F,F)\\
& =\frac{1}{2}\int_\Gamma\mu(d\gamma)\int_{\mathbb R^d}dx\int_{\mathbb R^d}dy\,\eps^{-d-2}
a((x-y)/\eps)r(x,\gamma)^{1/2} r(y,\gamma)^{1/2}(F(\gamma\cup x)-F(\gamma\cup y))^2\\
&=\frac{1}{2}\int_\Gamma\mu(d\gamma)\int_{\mathbb R^d}dx\int_{\mathbb R^d}dy\, a(y)r(x+\eps y,\gamma)^{1/2}
r(x,\gamma)^{1/2}\left(\frac{F(\gamma\cup\{x+\eps y\})-F(\gamma\cup x)}{\eps}\right)^2.
\end{align*}
Assume that $0<|\eps|\le1.$ Noting that the function $F$ is local
(i.e., there exists $\Delta\in\mathcal B_0(\mathbb R^d)$ such that $F(\gamma)=F(\gamma_\Delta)$ for all $\gamma\in \Gamma$) and that the function $a$ has a compact support, we conclude that there exists $\Lambda\in\mathcal B_0(\mathbb R^d)$ such that

\begin{equation}\label{gsaw4ty}
\begin{aligned}
\mathcal E_\eps(F,F)
& =\int_\Gamma\mu(d\gamma)\int_{\Lambda}dx\int_{\mathbb R^d}dy\, a(y)r(x+\eps y,\gamma)^{1/2}
r(x,\gamma)^{1/2}
\\
& \times \left(\frac{F(\gamma\cup\{x+\eps y\})-F(\gamma\cup x)}{\eps}\right)^2.
\end{aligned}
\end{equation}

By the dominated convergence theorem
 \begin{equation}\label{dse5ry}r(x,\gamma)^{1/2}\left(\frac{F(\gamma\cup\{x+\eps y\})-F(\gamma\cup x)}{\eps}\right)^2\to r(x,\gamma)^{1/2}\langle \nabla_x F(\gamma\cup x),y\ra^2 \end{equation} in
$L^2(\Gamma\times\Lambda\times\mathbb R^d,\mu(d\gamma)\,dx\,dy\, a(y))$ as $\eps\to0$.
By Lemma \ref{ufytcfdf} with $\alpha=1/2$, \eqref{gsaw4ty} and \eqref{dse5ry}
 \begin{equation}\label{hdstred}\mathcal E_\eps(F,F)\to \frac{1}{2}\int_\Gamma\mu(d\gamma)\int_{\Lambda}dx\int_{\mathbb R^d}dy\, a(y) r(x,\gamma)\langle \nabla_x F(\gamma\cup x),y\ra^2. \end{equation}
Since $a(y)=\tilde a(|y|)$, for any $i,j\in\{1,\dots,d\}$, $i\ne j$, we have
$$\int_{\mathbb R^d}a(y)y_iy_j\,dy=0$$
and $$c=\frac{1}{2}\int_{\mathbb R^d}a(y)y_i^2\,dy,\quad i=1,\dots,d.$$
Therefore, by \eqref{hdstred},
$$ \mathcal E_\eps(F,F)\to c\int_\Gamma\mu(d\gamma)\int_{\mathbb R^d}dx\, r(x,\gamma)|\nabla_xF(\gamma\cup x)|^2.$$
From here the theorem follows by the polarization identity.
\end{proof}

We will now show that the limiting form $(\mathcal E_0,\mathcal{F}C^\infty_{\mathrm b}(C^\infty_{0}(\mathbb R^d),
\Gamma))$ is closable and its closure identifies a diffusion process.

In what follows, we will assume that the conditions of Theorem~\ref{trsres} are satisfied.
We have
\begin{align*}
k(x,y)&=\int_{\mathbb R^d}\varkappa(x-u)\varkappa(y-u)\,du
\\
&=\int_{\mathbb R^d}\varkappa(u-y)\varkappa(u-x)\,du
=\int_{\mathbb R^d}\varkappa(u)\varkappa(u+y-x)\,du.
\end{align*}
Hence, by \eqref{tsrs5r}, the function $k(x,y)$ is  continuous  on $(\mathbb R^d)^2$. Thus, by Remark~\ref{cfdxrbv},
$(Y(x))_{x\in X}$ is a Gaussian random field and formula \eqref{36}
holds for all $(x,y)\in (\mathbb R^d)^2$.

Consider the semimetric
\begin{equation}\label{gdtrdjjjv}
\begin{aligned}
D(x,y):&=\frac12\Big(\int_\Omega(Y(x)-Y(y))^2\,d\mathbb P\Big)^{1/2}
\\
&=\frac12\big(k(x,x)+k(y,y)
-2k(x,y)\big)^{1/2}
\\
&=\Big( \int_{\mathbb R^d}\varkappa(u)\big(\varkappa(u)-\varkappa(u+y-x)\big)
\,du   \Big)^{1/2},
\quad x,y\in\mathbb R^d.
\end{aligned}
\end{equation}
The associated metric entropy $H(D,\delta)$ is defined as $H(D,\delta):=\log N(D,\delta)$, where $N(D,\delta)$ is the minimal number of points in a $\delta$-net in $B(0,1)=\{x\in\mathbb R^d\mid |x|\le1\}$ with respect to the semimetric  $D$, i.e., points $x_i$ such that the open balls centered at $x_i$ and of radius $\delta$ (with respect to $D$) cover $B(0,1)$. The expression
$$ J(D):=\int_0^1 \sqrt{H(D,\delta)}\,d\delta$$
is called the Dudley integral. The following result holds, see e.g.\ \cite[Corollary~7.1.4]{Bogachev} and the references therein.

\begin{theorem}\label{gtsre} Assume that $J(D)<\infty$. Then the Gaussian random field $(Y(x))_{x\in\mathbb R^d}$ has a continuous modification.
\end{theorem}

\begin{remark}\label{cdtrstrh}
Let $\varkappa$ be as in Remark~\ref{gdctsjh}. Then, by \eqref{gdtrdjjjv}, for any $x,y\in B(0,1)$
\begin{align*}
D(x,y)^{2}&\le\|\varkappa(\cdot)\|_{L^2(\mathbb R^d,dx)}\Big(\int_{\mathbb R^d}(\varkappa(u)-
\varkappa(u+y-x))^2\,du\Big)^{1/2}\\
&\le \|\varkappa(\cdot)\|_{L^2(\mathbb R^d,dx)}\|K(\cdot,2)\|_{L^2(\mathbb R^d,dx)}|y-x|,
\end{align*}
where we assumed that $K(\cdot,2)\in L^2(\mathbb{R}^d,dx)$.
Then $J(D)<\infty$, see e.g.\ \cite[Example~7.1.5]{Bogachev}.
\end{remark}

Denote by $\ddot\Gamma$ the space of all multiple configurations in $\mathbb R^d$. Thus, $\ddot\Gamma$ is the set of all Radon $\mathbb Z_+\cup\{+\infty\}$-valued measures on $\mathbb R^d$, In particular, $\Gamma\subset\ddot\Gamma$. Analogously to the case of $\Gamma$, we define the vague topology on $\ddot\Gamma$ and the corresponding Borel $\sigma$-algebra $\mathcal B(\ddot\Gamma)$.

\begin{theorem}\label{dse5duyd} Let $\varkappa(x,y)$ be of the form $\varkappa(x-y)$ for some $\varkappa\in L^2(\mathbb R^d,dx)$. Let $J(D)<\infty$. Let $l\in \mathbb N$ and  $c>0$. Then

(i) The bilinear form $(\mathcal E_0,\mathcal{F}C^\infty_{\mathrm b}(C^\infty_{0}(\mathbb R^d),\Gamma))$ defined by \eqref{sawawe} is closable on $L^2(\Gamma,\mu^{(l)})$ and its closure will be denoted by $(\mathcal E_0,D(\mathcal E_0))$.

(ii) There exists a conservative diffusion process
\begin{align*}
M^{0}=\left(\Omega^{0},
\mathcal{F}^{0},(\mathcal{F}^{0}_{t})_{t\geq
0},(\Theta_{t}^{0})_{t\geq 0}, (X^{0}(t))_{t\geq 0},
(P^{0}_{\gamma})_{\gamma\in \ddot\Gamma}\right)
\end{align*}
on $\ddot\Gamma$ which is properly associated with
$(\mathcal{E}_{0}, D(\mathcal{E}_{0}))$. In particular,
$M^{0}$ is $\mu^{(l)}$-symmetric and has $\mu^{(l)}$ as invariant measure. In the case $d\ge2$, the set $\ddot\Gamma\setminus\Gamma$ is $\mathcal E^0$-exceptional, so that $\ddot\Gamma$ may be replaced by with $\Gamma$ in the above statement.
\end{theorem}

\begin{proof} We again discuss only  the case $l=1$, omitting the upper index $(1)$. By \eqref{sawawe}, for any  $F,G\in \mathcal{F}C^\infty_{\mathrm b}(C^\infty_{0}(\mathbb R^d),\Gamma)$,
\begin{equation}
\begin{aligned}
\mathcal E_0(F,G)&=c\int_{\Omega\times\Gamma}\tilde{\mathbb P}(d\omega,d\gamma)\int_{\mathbb R^d}
dx\, \tilde {\mathbb E}(Y(x,\omega)^2\mid\mathcal F) \la \nabla_x F(\gamma\cup x),\nabla_x G(\gamma\cup x)\ra
\\
&=\int_{\Omega\times\Gamma}\tilde{\mathbb P}(d\omega,d\gamma)\int_{\mathbb R^d}
dx\, Y(x,\omega)^2
\\
&\times \la \nabla_x (F(\gamma\cup x)-F(\gamma)),\nabla_x (G(\gamma\cup x)-G(\gamma))\ra.
\end{aligned}
\end{equation}
Fix $(\omega,\gamma)\in\Omega\times\Gamma$.
Denote
$$ f(x):=F(\gamma\cup x)-F(\gamma), \quad g(x):=G(\gamma\cup x)-G(\gamma).$$
Clearly, $f,g\in C_0^\infty(\mathbb R^d)$.
In view of Theorem~\ref{gtsre}, $Y(x,\omega)^2$ is a continuous function of $x\in\mathbb R^d$.  Hence, by \cite[Theorem~6.2]{dSKR}, the bilinear form
$$ \mathcal E(f,g):=\int_{\mathbb R^d} \la \nabla f(x),\nabla g(x)\ra Y(x,\omega)^2dx,
\quad f,g\in C_0^\infty(\mathbb R^d),$$
is closable on $L^2(\mathbb R^d, |Y(x,\omega)|^2\,dx)$. Now the closability of $(\mathcal E_0,\mathcal{F}C^\infty_{\mathrm b}(C^\infty_{0}(\mathbb R^d),\Gamma))$  on $L^2(\Gamma,\mu^{(l)})$ follows by a straightforward generalization of the proof of \cite[Theorem~6.3]{dSKR}.
Part (ii) of the theorem can be shown completely analogously to \cite{MR2, RS}, see also \cite{KLR}.
\end{proof}

\begin{remark}
Heuristically, the generator of $(\mathcal E_0,D(\mathcal E_0))$ has the form
$$ (LF)(\gamma)=\sum_{x\in\gamma}\Big(\Delta_xF(\gamma)+\Big\la \frac{\nabla_xr(x,\gamma
\setminus x)}{r(x,\gamma\setminus x)}\,,\nabla_x F(\gamma)\Big\ra\Big).  $$
Here, for $x\in\gamma$, $\nabla_xF(\gamma):=\nabla_yF(\gamma\setminus x\cup y)\big|_{y=x}$ and analogously $\Delta_x$ is defined. However, we should not expect that $r(x,\gamma)$ is differentiable in $x$.
\end{remark}

{\it Acknowledgments.}
We are grateful to Alexei Daletskii, Dmitri Finkelshtein, Yuri Kondratiev, and Olexandr Kutoviy for many useful
discussions. EL acknowledges the financial support of the SFB~701 ``Spectral structures and topological
methods in mathematics'', Bielefeld University.  EL was partially supported by
  the International Joint Project grant 2008/R2 of the Royal Society and by
 the PTDC/MAT/67965/2006 grant, University of Madeira.

\end{document}